\documentclass[11pt]{article}
\usepackage{times}
\usepackage{amsmath,amsfonts,amstext,amssymb,amsbsy,amsopn,amsthm,eucal}
\usepackage{txfonts}
\usepackage[T1]{fontenc}
\usepackage[utf8]{inputenc}
\usepackage{dsfont}
\usepackage{graphicx}
\usepackage{hyperref}

\usepackage[english]{babel}
\usepackage{tocbibind}
\numberwithin{equation}{section}
\hypersetup{
  pdfpagelayout=SinglePage,
  pdftitle={Sharp estimates on the first eigenvalue of the p-Laplacian with negative Ricci lower bound},
  pdfauthor={Naber, Aaron and Valtorta, Daniele},
  pdfsubject={Sharp estimate on the first eigenvalue of the p-Laplacian on a compact Riemannian manifold with Ricci curvature bounded from below by a negative constant and (possibly) convex boundary and Neumann Boundary conditions},
  pdfkeywords={first eigenvalue, spectral gap, p-Laplacian, lower bound Ricci curvature negative},
  pdfpagemode=UseOutlines,
  colorlinks,
  linkcolor=[rgb]{0,0,0.7},
  urlcolor=[rgb]{0,0,0.4},
  citecolor=[rgb]{0.4,0.1,0}
  }

\newcommand{\R}{\mathbb{R}}

\newcommand{\ps}[2]{\left\langle#1\middle\vert#2\right\rangle}

\newcommand{\ton}[1]{\left(#1\right)}
\newcommand{\qua}[1]{\left[#1\right]}

\newcommand{\abs}[1]{\left|#1\right|}

\newcommand{\Ric}{\operatorname{Ric}}
\newcommand{\V}{\operatorname{Vol}}
\newcommand{\Vol}{\text{Vol}}

\newcommand{\pdw}{\dot w ^{(p-1)}}
\newcommand{\pw}{w^{(p-1)}}
\newcommand{\pG}{G^{(p-1 )}}
\newcommand{\pu}{ u^{(p-1)}}
\newcommand{\pf}{Pr\"{u}fer }

\newcommand{\sinp}{\operatorname{sin_p}}
\newcommand{\cosp}{\operatorname{cos_p}}

\newcommand{\dN}{\mathds{N}}

\newtheorem{prop}{Proposition}[section]
\newtheorem{proposition}[prop]{Proposition}

\newtheorem{theorem}[prop]{Theorem}

\newtheorem{definition}[prop]{Definition}
\newtheorem{lemma}[prop]{Lemma}

\newtheorem{remark}[prop]{Remark}

\newtheorem{corollary}[prop]{Corollary}

\begin{document}
\title{Sharp estimates on the first eigenvalue of the p-Laplacian with negative Ricci lower bound}
\author{Aaron Naber and Daniele Valtorta}

\date{\today}

\maketitle
\begin{abstract}
 We complete the picture of sharp eigenvalue estimates for the $p$-Laplacian on a compact manifold by providing {\it sharp} estimates on the first nonzero eigenvalue of the nonlinear operator $\Delta_p$ when the Ricci curvature is bounded from below by a negative constant. We assume that the boundary of the manifold is convex, and put Neumann boundary conditions on it. The proof is based on a refined gradient comparison technique and a careful analysis of the underlying model spaces.
\end{abstract}


\tableofcontents

\section{Introduction}
Let $M$ be an $n$-dimensional compact Riemannian manifold. For a function $u\in W^{1,p}(M)$, its $p$-Laplacian is defined by
\begin{gather}
  \Delta_p u\equiv \operatorname{div}(\abs{\nabla u}^{p-2} \nabla u)\ ,
\end{gather}
where the equality is in the weak $W^{1,p}(M)$ sense. Following standard convention, we will denote the first positive eigenvalue of this operator as $\lambda_{1,p}$, assuming Neumann boundary conditions if $\partial M \neq \emptyset$.  The number $\lambda_{1,p}$ may be characterized variationally in terms of a Poincar\'e inequality as the minimizer
\begin{gather*}
 \lambda_{1,p}=\inf\left\{\frac{\int_M \abs{\nabla u}^p d\V}{\int_M \abs u^p d\V} \ \text{ with }u\in M \ s.t. \ \ \int_M \abs u ^{p-2} u \ d\V=0\right\} \ .
\end{gather*}
In particular, using standard variational techniques we see that $\lambda_{1,p}$ is the smallest positive real number such that there exists a nonzero $u\in W^{1,p}(M)$ satisfying in the weak sense
\begin{gather}\label{eq_plap}
\begin{cases}
 \Delta_p(u)= -\lambda_{1,p} \abs u ^{p-2} u \ \ \ \ \ &on \ M\\
 \ps{\nabla u}{\hat n}=0 \ \ \  \ \ \ \ &on \ \partial M \, .
\end{cases}
\end{gather}
In this work  we prove the sharp estimate on $\lambda_{1,p}$ assuming that the Ricci curvature of the manifold is bounded below by $(n-1)k<0$. This completes the picture of the sharp estimates for $\lambda_{1,p}$. Indeed, for $k>0$, the generalized Obata's theorem has been obtained in \cite{matei}, while for $k=0$ the sharp estimate with characterization of equality has been proved in \cite{svelto}.

In the linear case, i.e., assuming $p=2$, the sharp estimate on $\lambda_{1,2}$ has been proved in \cite{kro} using a technique based on a refined gradient comparison theorem for an eigenfunction $u$ of the Laplacian.  Later, this technique was adapted by D. Bakry and Z. Qian in \cite{new} to obtain eigenvalue estimates for weighted linear Laplace operators with assumptions on the Bakry-Emery Ricci curvature.  A variational formula for $\lambda_{1,2}$ has been derived in \cite{chen} for a weighted Laplacian by using a coupling method, which also implies the sharp comparison theorem presented later in \cite{new}, see also \cite{chen-} for some earlier results proved with the same argument. 

A gradient comparison theorem in the spirit of \cite{kro} was first introduced for $p$-Laplace operators in \cite{svelto} under the assumption of nonnegative Ricci curvature, and here we introduce such a gradient comparison formula in Section \ref{s:GradientComparison} for $k<0$.  In each case one constructs an invertible one-dimensional function $w:[a,b]\subset \R \to \R$ satisfying a particular ODE.  If $u(M)\subset w[a,b]$, then it is possible to estimate
\begin{gather}\label{e:gradientestimate}
 \abs{\nabla u}(x)\leq \dot w \big(w^{-1} (u(x))\big)\, .
\end{gather}
The proof involves a careful use of a maximum principle.  In order to generalize to the nonlinear case the maximum principles needed for (\ref{e:gradientestimate}), we will use the linearized $p$-Laplace operator and a generalized $p$-Bochner formula (see Section \ref{sec_pboch} for the details).

After deriving an estimate of the form \eqref{e:gradientestimate} in Section \ref{s:GradientComparison}, the primary technical work of this paper is then to study the properties of the one dimensional models $w(t)$ and their relations to the properties on the eigenfunction $u$. In particular, to derive {\it sharp} estimates it is necessary to find a model $w$ such that $u(M)=w[a,b]$. In order to prove that this is always possible, we use a volume comparison argument described in Section \ref{sec_max}.

The study of the one-dimensional model will be carried out using a version of the \pf transformation introduced in Section \ref{sec_1d_neg}, a technique which will allow us to deal comfortably with the nonlinearity of the model equations (if $p\neq 2$). In particular, in Section \ref{s:diamter} we will use this tool and the convexity properties of our model equations to derive a lower bound on the diameters of the various model functions, defined as the distance between two consecutive zeros of $\dot w$.

Using a standard geodesic argument, an easy consequence of the gradient comparison theorem is a comparison between the diameter $d$ of the manifold $M$ and the diameter $\delta$ of the model function $w$. As we will see, $\delta$ depends monotonically on $\lambda$, thus the sharp estimate can be obtained by inverting the estimate on $d$ (see Section \ref{sec_main}).

It is worth mentioning that some nonsharp lower bounds for $\lambda_{1,p}$ have already been proved assuming negative lower bounds on the Ricci curvature. In \cite{matei}, the author obtains lower and upper bounds on $\lambda_{1,p}$ as a function of Cheeger's isoperimetric constant (see in particular \cite[Theorems 4.1, 4.2 and 4.3]{matei}). Among others, the explicit (non sharp) lower bound $\lambda_{1,p} \geq C_1(n,k,p) \operatorname{exp}[-(1+C_2(n,k,p) d)] d^{-p}$, has been proved by W. Lin-Feng in \cite{wang}. It is also worth mentioning that in the Euclidean setting, using a completely different approach based on variational techniques, a sharp estimate on the $p$-Poincar\'e constant is obtained in \cite{carlo}, see also \cite{carlo2} for the weighted case.

To state the main theorem let us first state the relevant one dimensional models that will be used as a comparison tool.

\begin{definition}\label{deph_l}
For each $k<0$, $0<d<\infty$ and $n\in \dN$ let $\bar \lambda(n,k,d)$ denote the first positive Neumann eigenvalue on $[-d/2,d/2]$ of the eigenvalue problem
 \begin{gather*}
  \frac{d}{dt}\ton{\dot w^{(p-1)}} + (n-1)\sqrt{-k}\tanh\ton{\sqrt{-k} t} \dot w ^{(p-1)}+\bar \lambda(n,k,d) w^{(p-1)}=0\, .
 \end{gather*}
\end{definition}
\begin{remark}\label{rem_l}
We will prove that this is equivalent to finding the unique value of $\bar \lambda$ such that the solution of
\begin{gather*}
 \begin{cases}
  \dot \phi = \ton{\frac{\bar \lambda}{p-1}}^{1/p} + \frac{(n-1)\sqrt{-k}}{p-1}\tanh\ton{\sqrt{-k} t}  \cosp^{(p-1)}(\phi) \sinp(\phi)\\
  \phi(0)=0
 \end{cases}
\end{gather*}
satisfies $\phi(d/2) = \pi_p/2$, where $\sinp$, $\cosp$ and $\pi_p$ are defined in Section \ref{sec_1d_neg}.
\end{remark}
\begin{remark}
An expression for $\bar \lambda(n,k,d)$ in terms of elementary functions is not clear.
\end{remark}

The main theorem in this work is now following.
\begin{theorem}
 Let $M$ be an $n$-dimensional complete Riemannian manifold with Ricci curvature bounded from below by $(n-1)k \leq 0$, diameter $d<\infty$ and with possibly empty convex $C^2$ boundary.  Then we have the sharp estimate:
 \begin{gather*}
  \lambda_{1,p}\geq \bar \lambda(n,k,d)\, .
 \end{gather*}
\end{theorem}

Unlike the $k\geq 0$ case, the sharp lower bound $\bar \lambda(n,k,d)$ is never attained.  To see sharpness of this result we build a sequence of Riemannian manifolds $M_i$ with $\Ric\geq (n-1)k$ and $\text{diam}(M_i)=d_i\searrow d$ such that $\lambda_{1,p}(M_i)\leq\bar \lambda(n,k,d)$.  The smooth Riemannian manifolds $M_i$ are all warped products with smooth boundary, and as $i\to \infty$ we see that $\Vol(M_i)\to 0$ with the $M_i$ collapsing geometrically to the one dimensional interval $[-d/2,d/2]$. 

\begin{remark}\rm
 As mentioned in the introduction, the sharp lower bound for the case $k>0$ and $d=\pi/\sqrt k$ has been obtained in \cite{matei}. In particular, the author proves that, under these assumptions, $\lambda_{1,p}$ is bounded below by the first eigenvalue of the $n$-dimensional sphere of diameter $d$.
 
 It is easy to adapt the techniques used in this paper to give an alternative proof of this result. In particular, if $k>0$, for all $d\leq \pi/\sqrt k$ one can prove that $\lambda_{1,p}\geq \lambda(n,k,d)$, where in this case $\lambda(n,k,d)$ is determined by the one dimensional problem obtained by replacing $\sqrt{-k}\tanh(\sqrt{-k} t)$ with $\sqrt k \tan(\sqrt k t)$ in Definition \ref{deph_l} and Remark \ref{rem_l}.
\end{remark}

\section{Notation}
Throughout the article, we use the following notation. We have that $M$ denotes a compact Riemannian manifold of dimension $n$ whose Ricci curvature satisfies the lower bound
\begin{align}
\Ric\geq (n-1)k\, ,
\end{align}
where $k<0$.  We denote the diameter of $M$ by
\begin{align}
d\equiv \text{diam}(M)\, .
\end{align}

In the case that the boundary $\partial M$ is nonempty, we always assume it is (nonstrictly) convex. For any $p\in (1,\infty)$, we use the standard convention
\begin{gather*}
 u^{(p-1)}\equiv \abs u ^{p-2} u = \abs{u}^{p-1} \operatorname{sign} (u)\, .
\end{gather*}
We denote the Hessian of a function $u:M\to \R$ by $H_u$ and set
\begin{gather*}
 A_u = \frac{H_u\ton{\nabla u,\nabla  u}}{\abs{\nabla u}^2}
\end{gather*}
where $\nabla u \neq 0$.\\
We denote by $u$ a nonconstant solution of:
\begin{gather*}
 \Delta_p(u)=-\lambda_{1,p} u^{(p-1)}
\end{gather*}
with Neumann boundary conditions if necessary. We also define 
\begin{align}
u^\star \equiv \max\{u\}\, .
\end{align}

Recall that for some $\alpha>0$ we have that $u\in C^{1,\alpha}(M)\cap W^{1,p}(M)$, and elliptic regularity ensures that $u$ is a smooth function where $\nabla u \neq 0$ and $u\neq 0$. If $\nabla u(x)\neq 0$ and $u(x)=0$, then $u\in C^{3,\alpha}(U)$ if $p>2$ and $u\in C^{2,\alpha}(U)$ for $1<p<2$, where $U$ is a suitably small neighborhood of $x$. The standard reference for these results is \cite{reg}, where the problem is studied in local coordinates.

Regarding regularity issues, it is worth mentioning the very recent article \cite{teix}, which obtains studies the behaviour of the exponent $\alpha$ at points where $\nabla u =0$.

By an easy application of the divergence theorem, it is easy to see that
\begin{gather*}
 \int_M u^{(p-1)} d\Vol = -\lambda_{1,p}^{-1} \int_M \Delta_p (u) d\Vol = 0\, .
 \end{gather*}
Thus, without loss of generality, we rescale $u$ in such a way that
\begin{gather}
 \abs u \leq 1 \quad \quad \min\{u\}=-1 \quad \quad 0<\max\{u\}={u^\star}\leq 1\, .
\end{gather}

\section{Linearized p-Laplacian and p-Bochner formula}\label{sec_pboch}
In this section, we briefly recall some properties of the linearized $p$-Laplacian and a generalized version of the Bochner formula.

We start with the definition of the linearized $p$-Laplace operator. Quite common in recent literature, this operator has been used, for example, in \cite{kn,hui,svelto}.
\begin{definition}
 Given two functions $u,\eta$ we define:
\begin{gather*}
 P_u(\eta)\equiv \operatorname{div}\ton{(p-2) \abs{\nabla u}^{p-4}\ps{\nabla u}{\nabla \eta}\nabla u + \abs{\nabla u }^{p-2}\nabla \eta}=\\
= \abs{\nabla u}^{p-2}\Delta \eta +(p-2)\abs{\nabla u}^{p-4}H_\eta\ton{\nabla u,\nabla u}+ (p-2)\Delta_p(u)\frac{\ps{\nabla u}{\nabla \eta}}{\abs{\nabla u }^2}+ \\
+ 2(p-2)\abs{\nabla u}^{p-4}H_u\ton{\nabla u,\nabla \eta - \frac{\nabla u}{\abs{\nabla u}}\ps{\frac{\nabla u}{\abs{\nabla u}}}{\nabla \eta}} \ ,
\end{gather*} 
whenever $|\nabla u|\neq 0$.
\end{definition}
If $u$ is an eigenfunction of the $p$-Laplacian, this operator is defined pointwise only where the gradient of $u$ is non zero (and so $u$ is smooth in a neighborhood of the point) and it is easily proved that at these points it is strictly elliptic. For convenience, denote by $P^{II}_u$ the second order part of $P_u$, which is
\begin{gather}\label{deph_pu}
 {P_u}^{II}(\eta) \equiv \abs{\nabla u}^{p-2}\Delta \eta  +(p-2)\abs{\nabla u}^{p-4}H_\eta\ton{\nabla u, \nabla u} . 
\end{gather}

We cite from \cite{svelto} the following version of the Bochner formula.
\begin{proposition}[p-\textsc{Bochner formula}]\label{prop_pboch}
 Given $x\in M$, a domain $U$ containing $x$, and a function $u\in C^3(U)$, if $\nabla u \neq 0$ on $U$ we have
\begin{gather*}
 \frac{1}{p} P^{II}_u(\abs{\nabla u}^{p})=\abs{\nabla u}^{2(p-2)}\Big\{\abs{\nabla u}^{2-p}\Big[\ps{\nabla \Delta_p u}{\nabla u}-(p-2)A_u\Delta_p u\Big]+\abs{H_u}^2+p(p-2)A_u^2+ \operatorname{Ric}(\nabla u,\nabla u) \Big\} \ .
\end{gather*}
In particular this equality holds if $u$ is an eigenfunction of the $p$-Laplacian and $\nabla u |_x \neq 0$ and $u(x)\neq 0$. If $p\geq 2$, this results holds also where $u(x)=0$.
\end{proposition}

In order to estimate $P^{II}_u(\abs{\nabla u}^p)$ from below, we also recall the following generalization of the curvature dimension inequality available for the Hessian of a smooth function (again, see \cite{svelto} for the details).
\begin{corollary}\label{cor_n}
For every $n\leq n' \in \R$, and for every point where $\nabla u \neq 0$ and $u\in C^2$, we have
\begin{gather*}
 \abs{\nabla u}^{2p-4}\ton{\abs{H_u}^2 + p(p-2)A_u^2}\geq\\
\geq \frac{(\Delta_p u)^2}{n'} + \frac{n'}{n'-1}\ton{\frac{\Delta_p u}{n' } -(p-1)\abs{\nabla u}^{p-2}A_u}^2 \ .
\end{gather*}
\end{corollary}

\section{One Dimensional Model}\label{sec_1d_neg}
In this section we introduce the one dimensional model functions that will be used in subsequent sections as a comparison for the eigenfunctions $u$ of the $p$-Laplacian.

For $n$, $k<0$ fixed define for $i=1,2,3$ the nonnegative functions $\tau_i$ on $I_i\subset \R$ by:
\begin{enumerate}
 \item $\tau_1 (t)=\operatorname{sinh}\ton{\sqrt {-k}\ t}$, defined  on $I_1=[0,\infty)$,
 \item $\tau_2 (t)=\operatorname{exp}\ton{\sqrt {-k}\ t}$ on $I_2=\R$ ,
 \item $\tau_3 (t)=\operatorname{cosh}\ton{\sqrt {-k}\ t}$ on $I_3=\R$ ,
\end{enumerate}
and let $\mu_i= \tau_i^{n-1}$.  For each $\tau_i$ and each $0< \epsilon\leq 1$ we can consider the Riemannian manifold defined by the warped product
\begin{align}
 M=[a,b]\times_{ \epsilon \tau_i} S^{n-1}\, ,
\end{align}
where the metric is given by
\begin{align}
g_M\equiv dr^2 +\epsilon^2\tau_i^2 g_{S^{n-1}}\, .
\end{align}
Let $(t,x), \ t\in [a,b], \ x\in S^{n-1}$ denote the product coordinates.  By some relatively standard computations \footnote{for the details, see for example \cite{petersen,milman}}, $M$ is a manifold whose Ricci curvature satisfies
\begin{align}
\Ric\geq (n-1)k\, \notag\\
\Ric(\partial_t,\partial_t) = (n-1)k\, ,
\end{align}
with $\mu_i$ measuring the volume of the radial slices. Indeed
\begin{gather}
 \Vol([c,d]\times_{\epsilon \tau_i} S^{n-1}) =\epsilon^{n-1}\Vol(S^{n-1}) \int_{c}^d  \mu_i(t) dt\, .
\end{gather}

Note that $[0,d]\times_{\tau_1} S^{n-1}$ is nonother than the geodesic ball of radius $d$ in the hyperbolic space.  Now let $T_i=-\frac{\dot \mu_i}{\mu_i}$, that is:
\begin{enumerate}
 \item $T_1 (t)=-(n-1)\sqrt{-k}\operatorname{cotanh}\ton{\sqrt{-k}\ t}$, defined on $I_1=(0,\infty)$,
 \item $T_2 (t)=-(n-1)\sqrt{-k}$, defined on $I_2=\R$,
 \item $T_3 (t)= -(n-1)\sqrt{-k}\operatorname{tanh}\ton{\sqrt{-k}\ t}$, defined on $I_3=\R$.
\end{enumerate}
Note that all functions $T_i$ satisfy
\begin{gather}\label{eq_T}
 \dot T= \frac{T^2}{n-1} + (n-1)k\, .
\end{gather}

Now we are ready to introduce our one dimensional model functions.
\begin{definition}\label{deph_1dm}
 Fix $\lambda>0$. Define the function $w=w^{p,\lambda}_{k,n,i,a}$ to be the solution to the initial value problem on $I_i$:
 \begin{gather}\label{eq_1dm_neg}
  \begin{cases}
   \frac{d}{dt} \dot w ^{(p-1)} - T_i \dot w ^{(p-1)} + \lambda w^{(p-1)} =0\\
   w(a)=-1 \quad \dot w (a)=0
  \end{cases}
 \end{gather}
where $a\in I_i$. Equivalently, $w^{p,\lambda}_{k,n,i,a}$ are the solutions to:
 \begin{gather}\label{eq_1dm_negmu}
  \begin{cases}
   \frac{d}{dt}\ton{\mu_i \dot w ^{(p-1)}} +\lambda \mu_i w^{(p-1)} =0\\
   w(a)=-1 \quad \dot w (a)=0
  \end{cases}
 \end{gather}
\end{definition}
\begin{remark}
\rm When some of the parameters $\lambda$, $p$, $k$, $n$, $i$, $a$ are fixed and there is no risk of confusion, we may often omit them.
\end{remark}

\begin{remark}\label{rem_w}
\rm Define on $M=[a,b]\times_{\tau_i} S^{n-1}$ the function $u(t,x)=w(t)$. It is easy to realize that $u$ solves the eigenvalue equation $\Delta_p(u)+\lambda u^{(p-1)}=0$ on $M$. Moreover, if $\dot w(b)=0$, then $u$ has Neumann boundary conditions on such a manifold.
\end{remark}

Now we prove existence, uniqueness and continuous dependence with respect to the parameters for the solutions of the IVP \eqref{eq_1dm_neg}. In order to do so, we introduce a version of the so-called \pf transformation (similar transformations are well-studied in nonlinear ODE theory, see for example \cite[Section 1.1.3]{dosly}). In a sense, we put $p$-polar coordinates on the phase plane $(w,\dot w)$ of the function $w$. 

Here we briefly recall the definition of the functions $\sinp$ and $\cosp$ (for more detailed references, see for example \cite[Section 1.1.2]{dosly} or \cite[Section 2]{svelto}).
\begin{definition}
For every $p\in (1,\infty)$, define the positive number $\pi_p$ by:
\begin{gather}
 \pi_p = \int_{-1}^1 \frac{ds}{\ton{1-s^p}^{1/p}} = \frac{2\pi}{p\sin(\pi/p)} \, .
\end{gather}
The $C^1(\R)$ function $\sinp:\R\to [-1,1]$ is defined implicitly on $[-\pi_p/2,3\pi_p/2]$ by:
\begin{gather*}
 \begin{cases}
t=\int_0^{\sin_p(t)} \frac{ds}{(1-s^p)^{1/p}} & \text{  if } t\in \qua{-\frac{\pi_p}{2},\frac{\pi_p}{2}}\\
\sinp(t) = \sinp(\pi_p-t) & \text{  if } t\in \qua{\frac{\pi_p}{2},\frac{3 \pi_p}{2}}
 \end{cases}
\end{gather*}
and is periodic on $\R$. Set also by definition $\cosp(x) = \frac d {dt} \sinp(t)$. The usual fundamental trigonometric identity can be generalized by:
\begin{gather*}
 \abs{\sinp(t)}^p + \abs{\cosp(t)}^p =1\, ,
\end{gather*}
and so it is easily seen that $\cosp^{(p-1)}(t)\in C^1(\R)$.
\end{definition}

\begin{definition}\label{deph_pf}
Let $\alpha= \ton{\frac{\lambda}{p-1}}^{1/p}$ and fix some $w=w_{k,n,i,a}$. Define the functions $e=e_{k,n,i,a}\geq 0$ and $\phi=\phi_{k,n,i,a}$ by:
\begin{gather}\label{eq_wpf}
 \alpha w = e \sinp (\phi) \ \ \ \dot w =e \cosp(\phi)\, ,
\end{gather}
or equivalently:
\begin{gather*}
  e\equiv \ton{\dot w ^p + \alpha^p w^p}^{1/p}\ \ \ \ \ \phi\equiv \operatorname{arctan_p} \ton{\frac{\alpha w}{\dot w}}\, .
\end{gather*}
Note that the variable $\phi$ is well-defined up to $\pi_p$ translations. 
\end{definition}

Let $w$ satisfy \eqref{eq_1dm_neg}. Differentiating, substituting and using equation \eqref{eq_1dm_neg} we get that $\phi$ and $e$ satisfy the following first order IVPs:
\begin{gather}\label{eq_pf_neg}
 \begin{cases}
  \dot \phi = \alpha - \frac{T}{p-1}\cosp^{p-1} (\phi)\sinp(\phi) \\
  \phi(a)_{\operatorname{mod} \pi_p}=-\frac{\pi_p}2  
 \end{cases}\\
 \begin{cases}\label{eq_pf_nege}
  \frac d {dt} \log(e) = \frac{\dot e}{e} =\frac{T}{(p-1)}\cosp^{p} (\phi)\\
  e(a)=\alpha
 \end{cases}
\end{gather}

Since both $\sinp$ and $(p-1)^{-1}\cosp^{p-1}$ are Lipschitz functions with Lipschitz constant $1$, it is easy to apply Cauchy's theorem and prove existence, uniqueness and continuous dependence on the parameters. Indeed, we have the following:
\begin{proposition}\label{prop_exun1}
 If $T=T_2,T_3$, for any $a\in \R$ there exists a unique solution to \eqref{eq_1dm_neg} defined on all $\R$. The solution $w$ is of class $C^1(\R)$ with $\dot w^{(p-1)}\in C^1(\R)$ as well. Moreover, the solution depends continuously on the parameters $a$, $n\in \R$ and $k<0$ in the sense of local uniform convergence of $w$ and $\dot w$ in $\R$.
\end{proposition}
The same argument work verbatim if $T=T_1$ as long as $a>0$, while the boundary case deserves some more attention. However, using standard ODE techniques, also in this case it is possible to prove existence, uniqueness and continuous dependence for the solution of \eqref{eq_1dm_neg}. Although with a different model function $T$, a similar argument is carried out for example in \cite[Section 3]{W}. Thus we have the following Proposition.
\begin{proposition}\label{prop_exun2}
 If $T=T_1$, for any $a>0$ there exists a unique solution to \eqref{eq_1dm_negmu} defined (at least) on $(0,\infty)$. The solution $w$ is of class $C^1(0,\infty)$ with $\dot w^{(p-1)}\in C^(0,\infty)$ as well.  
 
 Also if $a=0$, the solution is unique and belongs to $C^1[0,\infty)$. 
 Moreover, the solution depends continuously on the parameters $a\geq 0$, $\lambda>0$, $n\geq 1$ and $k<0$ in the sense of local uniform convergence of $w$ and $\dot w$ in $(0,\infty)$.
\end{proposition}

\section{Gradient Comparison}\label{s:GradientComparison}
With the definitions given in the previous sections, we are ready to state and prove the gradient comparison theorem for the eigenfunction $u$. Although more technically involved, because we are in a nonlinear setting and need the use of the linearized $p$-Laplace operator, the proof of this result is similar to the proof of \cite[Theorem 1]{kro}.
\begin{theorem}[\textsc{Gradient comparison Theorem}]\label{grad_est_neg}
 Let $M$ be a compact $n$-dimensional Riemannian manifold with Ricci curvature bounded from below by $(n-1)k<0$, and possibly with $C^2$ convex boundary. Let $u$ be a solution to
 \begin{gather}
  \Delta_p (u) = -\lambda u^{(p-1)}
 \end{gather}
rescaled in such a way that $-1=\min\{u\} <0<\max\{u\}\leq 1$. Let $w$ be a solution of the one dimensional initial value problem:
\begin{gather}
\begin{cases}
 \frac d {dt}\dot w ^{(p-1)} - T \dot w ^{(p-1)} +\lambda w^{(p-1)}=0\\
 w(a)=-1 \quad \dot w(a)=0
\end{cases}
\end{gather}
where $T$ satisfies \eqref{eq_T}. Consider an interval $[a,b]$ in which $\dot w \geq 0$. If $$[\min(u),\max(u)]\subset [-1,w(b)]\, ,$$ then:
\begin{gather*}
\abs{\nabla u (x)}\leq \dot w(w^{-1}(u(x)))
\end{gather*}
for all $x\in M$.
\end{theorem}

\begin{proof}
Suppose for the moment that $\partial M$ is empty; the modification needed for the general case will be discussed in Remark \ref{rem_bou}.

In order to avoid problems at the boundary of $[a,b]$, we assume that 
\begin{gather*}
[\min\{u\},\max\{u\}]\subset (-1,w(b))\, , 
\end{gather*}
so that we have to study our one dimensional model only on compact subintervals of $(a,b)$, where $\dot w$ is bounded below by a positive constant. Since $\min\{u\}=-1,\, \max\{u\}>0$, we can obtain this by multiplying $u$ by a positive constant $\xi<1$. If we let $\xi\to 1$, then the original statement is proved.

Using the notation introduced in Section \ref{sec_1d_neg}, we define the family of functions on $M$:
\begin{gather}
 F_{K,c}\equiv \abs{\nabla u}^p - (c\dot w)^p|_{(cw)^{-1} u(x)}\, ,
\end{gather}
for $c\geq 1$ \footnote{note that, since $\min\{u\}=-\xi\simeq -1$, $F$ is not well defined for all $c<1$} and $K<0$. Since $w_{K,n,i,a}$ depends continuously in the $C^1$ sense on $K$, these functions are well-defined and continuous on $M$ if $K$ is sufficiently close to $k$.

In the following, we consider $i$, $a$, $\lambda$ and $n$ to be fixed parameters, while we will need to let $K$ vary in a neighborhood of $k$.

Using a contradiction argument, we prove that for every $\epsilon>0$ sufficiently small, $F_{k-\epsilon,1}\leq 0$ on all of $M$.

Define $\bar F_{k-\epsilon,c}=\max\{F_{k-\epsilon,c}(x), \ x\in M\}$, and suppose by contradiction that $\bar F_{k-\epsilon,1}>0$. Since
\begin{gather}
 \lim_{c\to \infty} \bar F_{k-\epsilon,c} =-\infty\, ,
\end{gather}
there exists a $\bar c \geq 1$ such that $\bar F_{k-\epsilon,\bar c}=F_{k-\epsilon, \bar c}(\bar x)=0$. It is clear that, at $\bar x$, $\abs{\nabla u }> 0$.

Hereafter, we will assume that $u$ is a $C^3$ function in a neighborhood of $\bar x$, so that $F$ will be a $C^2$ function in a neighborhood of this point. This is certainly the case if $u(\bar x)\neq 0$, or if $p\geq 2$. If $1<p<2$ and $u(\bar x)=0$, then $u$ has only $C^{2,\alpha}$ regularity around $\bar x$. However, this regularity issue is easily solved, as we will see in Remak \ref{rem_reg}.

Since we are assuming $\partial M=\emptyset$, $\bar x$ is internal maximum point, and thus
\begin{gather}\label{eq_max1}
 \nabla F_{ k-\epsilon, \bar c}(\bar x)=0 \\
 P^{II}_u (F_{ k-\epsilon, \bar c})(\bar x)\leq 0\label{eq_max2}
\end{gather}
Simple algebraic manipulations on equation \eqref{eq_max1} yield to the following relations valid at $\bar x$:
\begin{align*}
 p\abs{\nabla u}^{p-2} H_u\nabla u & = \frac p {p-1} \Delta_p (\bar cw) \nabla u\, ,\\
\abs{\nabla u}^{p-2} A_u = \abs{\nabla u}^{p-2} \frac{H_u\ton{\nabla u,\nabla u}}{\abs{\nabla u}^2}&= \frac{1}{p-1}\Delta_p (\bar cw)\, .
\end{align*}
Using Proposition \ref{prop_pboch} and Corollary \ref{cor_n} to estimate the left hand side of inequality \eqref{eq_max2}, we get that, at $\bar x$:
\begin{gather}
0\geq\frac 1 p P^{II}_u(F_{ k-\epsilon, \bar c})\geq -\lambda(p-1) \abs u ^{p-2} \abs{\nabla u}^p + (p-2)\lambda u^{p-1} \abs{\nabla u}^{p-2}A_u+\notag \\
+\frac{\lambda^2 \abs u^{2p-2}}{n} + \frac n {n-1}\ton{\frac{\lambda u^{p-1}}{n} - (p-1)\abs{\nabla u}^{p-2} A_u}^2 + (n-1)k\abs{\nabla u}^{2p-2}+\notag \\
+\frac{\lambda u^{p-1}}{p-1} \Delta_p (\bar c w)|_{(\bar cw)^{-1} (u)} -\abs{\nabla u}^p \frac{1}{\bar c\dot w} \left.\frac {d(\Delta_p(\bar c w))} {dt}\right\vert_{(\bar cw)^{-1} (u)} \label{eq_last}\, .
\end{gather}
At $\bar x$, $\abs{\nabla u }^p = (\bar c\dot w)^p |_{(\bar cw)^{-1} (u)}$, thus we obtain that, at $\bar t=(\bar c w )^{-1} (u(\bar x))$:
\begin{gather*}
0\geq-\lambda(p-1) \abs{\bar c w}^{p-2} (\bar c\dot w )^p+ \frac{p-2}{p-1}\lambda (\bar c w)^{p-1}\Delta_p(\bar c w)+\\
+\frac{\lambda^2 \abs{\bar c w}^{2p-2}}{n} + \frac n {n-1}\ton{\frac{\lambda (\bar c w)^{p-1}}{n} - \Delta_p (\bar c w)}^2 + (n-1)k(\bar c \dot w)^{2p-2}+\\
+\frac{\lambda (\bar c w)^{p-1}}{p-1} \Delta_p (\bar cw)- (\bar c \dot w )^p\frac{1}{\bar c\dot w} \frac {d(\Delta_p (\bar cw))} {dt}\, .
\end{gather*}
By direct calculation, using the ODE \eqref{eq_1dm_neg} satisfied by $\bar cw$, this inequality is equivalent to:
\begin{gather*}
(n-1)(k-K)(\bar c \dot w)^{2p-2}|_{\bar t}=(n-1)\epsilon (\bar c \dot w)^{2p-2}|_{\bar t}\leq 0\, ,
\end{gather*}
which is a contradiction. 
\end{proof}

\begin{remark}\label{rem_bou}
\rm{Analyzing the case with boundary, the only difference in the proof of the gradient comparison is that the point $\bar x$ may lie in the boundary of $M$, and so it is not immediate to obtain equation \eqref{eq_max1}. However, once this equation is proved, it is evident that $P^{II}_u F|_{\bar x}\leq 0$ and the rest of proof proceeds as before. In order to prove that $\bar x$ is actually a stationary point for $F$, the (nonstrict) convexity of the boundary is crucial. Using a technique similar to the proof of \cite[Theorem 8]{new}, we prove the following Lemma.}
\end{remark}
\begin{lemma}
If $\partial M $ is non empty, even if $\bar x \in \partial M$ the equation
\begin{gather}
 \nabla F_{k-\epsilon, \bar c}|_{\bar x}=0 
\end{gather}
remains valid.
\end{lemma}
\begin{proof}
Let $\hat n$ be the outward normal derivative of $\partial M$.

Since $\bar x$ is a point of maximum for $F_{k-\epsilon, \bar c}$, we know that all the derivatives of $F$ along the boundary vanish, and that the normal derivative of $F$ is nonnegative
\begin{gather*}
 \ps{\nabla F}{\hat n}\geq 0 \ .
\end{gather*}
Neumann boundary conditions on $\Delta_p$ ensure that $\ps{\nabla u}{\hat n}=0$. Define for simplicity $\psi(x)=(\bar c \dot w )^p|_{(\bar c w)^{-1}(x)}$. By direct calculation we have
\begin{gather*}
 \ps{\nabla F}{\hat n} = -\dot \psi|_{u(\bar x)} \ps{\nabla u}{\hat n} + p\abs{\nabla u}^{p-2} H_u(\nabla u,\hat n)=  p\abs{\nabla u}^{p-2} H_u(\nabla u,\hat n) \ .
\end{gather*}
Using the definition of second fundamental form $II(\cdot,\cdot)$ and the convexity of $\partial M$, we can conclude that
\begin{gather*}
 0\leq \ps{\nabla F}{\hat n} = p\abs{\nabla u}^{p-2} H_u(\nabla u,\hat n) = - p\abs{\nabla u}^{p-2} II(\nabla u,\nabla u)\leq 0 \, .
\end{gather*}
\end{proof}

\begin{remark}\label{rem_m_neg}
\rm Note that Corollary \ref{cor_n} is valid for all real $n'\geq n$, and so also the gradient comparison remains valid if we use model equations with ``dimension'' $n'$, i.e., if we assume $\dot T = \frac{T^2}{n'-1} +(n'-1)k$.
\end{remark}

\begin{remark}\label{rem_reg}
\rm As mentioned before, in case $1<p<2$ and $u(x)=0$, we have a regularity issue to address in the proof of the gradient comparison theorem. Indeed, in this case $F$ is only a $C^{1,\alpha}$ function and Equation \eqref{eq_max2} is not well-defined since there are two diverging terms in this equation. As it can be seen from \eqref{eq_last}, these terms are
\begin{gather*}
-\lambda(p-1) \abs u ^{p-2} \abs{\nabla u}^p\quad \text{ and } \quad -\abs{\nabla u}^p \frac{1}{\bar c\dot w} \left.\frac {d(-\lambda (\bar c w)^{p-1})} {dt}\right\vert_{(\bar cw)^{-1} (u)} \, .
\end{gather*}
 
 However, since $\nabla u(\bar x)\neq 0$, there exists an open neighborhood $U$ of $\bar x$ such that $U\setminus\{u=0\}$ is open and dense in $U$. On this set, it is easy to see that these two terms exactly cancel each other, and all the other terms in $P^{II}_u (F)$ are well-defined and continuous on $U$. Thus Equation \eqref{eq_max2} is valid even in this low-regularity context.
\end{remark}

It is not difficult to adapt the proof of the previous Theorem in order to compare different functions $w_{k,n,i,a}$. In particular, we can state the following:
\begin{theorem}\label{grad_est_negw}
For $j=1,2$ let $w_j=w_{k,n,i_j,a_j}$ be solutions to the one dimensional IVP \eqref{eq_1dm_neg} and let $b_j<\infty$ be the first point $b_j>a_j$ such that $\dot w_j(b_j)=0$. If
\begin{gather}
 w_1[a_1,b_1]\subset w_2[a_2,b_2]\, ,
\end{gather}
then we have the following comparison for the derivatives:
\begin{gather}
 \abs{\dot w_1}|_t \leq \dot w_2|_{w_2^{-1}(w_1(t))}\, ,
\end{gather}
or equivalently:
\begin{gather}
 \abs{\dot w_1}|_{w_1^{-1}(s)}\leq \abs{\dot w_2}|_{w_2^{-1}(s)}\, .
\end{gather}
\end{theorem}
\begin{proof}
 This Theorem can be proved directly using a method similar to the one described in the proof of Theorem \ref{grad_est_neg}. Another method is to define on $M=[a_1,b_1]\times_{\tau_i} S^{n-1}$ the function $u(t,x)=w_1(t)$, and use directly Theorem \ref{grad_est_neg} to get the conclusion. Note that $M$ might have nonconvex boundary in this case, but since $u(t,x)$ depends only on $t$, it is easy to find a replacement for Remark \ref{rem_bou}.
\end{proof}

\section{Fine properties of the one dimensional model}\label{sec_1drev}
In this section we study some fine properties of our one dimensional model. In particular, we study the oscillatory behaviour of the functions $w$ depending on $\lambda, \, i$ and $a$.  Throughout this section, $n$ and $k$ are fixed, and as usual we set $\alpha= \ton{\frac \lambda {p-1}}^{\frac 1 p}$.

To begin with, it is easy to see that in the model $i=3$ there always exists an odd solution $w_{3,-\bar a}$ which has maximum and minimum equal to $1$ in absolute value.
\begin{proposition}\label{prop_barasym}
 Fix $\alpha>0$, $n\geq 1$ and $k<0$. Then there always exists a unique $\bar a>0$ such that the solution $w_{3,-\bar a}=w^{p,\lambda}_{k,n,3,-\bar a}$ to the IVP \eqref{eq_1dm_neg} (with $T=T_3$) is odd. In particular, $w_{3,-\bar a}$ restricted to $[-\bar a,\bar a]$ has nonnegative derivative and has maximum equal to $1$.
\end{proposition}
\begin{proof}
 We use the \pf transformation to prove this theorem. For the sake of simplicity, here we write $\phi$ for $\phi_{i,a}$. Consider the IVP:
 \begin{gather}
  \begin{cases}
     \dot \phi = \alpha - \frac{T_3}{p-1}\cosp^{p-1}(\phi)\sinp(\phi)\\
     \phi(0)=0
  \end{cases}
 \end{gather}
Recall that $T_3(0)=0$, $T_3$ is odd and it is negative on $(0,\infty)$. By uniqueness of the solution, also $\phi$ is an odd function. Moreover, it is easily seen that as long as $\phi\in [-\pi_p/2,\pi_p/2]$, $\dot \phi \geq \alpha$.

This implies that there exists a $-\bar a \in [-\pi_p/(2\alpha),0]$ such that $\phi(-\bar a)=-\pi_p/2$. It is also easy to see that the corresponding solution $e(t)$ to equation \eqref{eq_pf_nege} is even, regardless of the value of $e(0)$.  Thus we have proved all the properties we were seeking for $w_{3,-\bar a}$.
\end{proof}

This proposition proves that we can always use the gradient comparison Theorem with $w_{3,-\bar a}$ as a model function. However, as we will see in the following section, to get a sharp estimate on the eigenvalue we will need a model function $w$ such that $\min\{w\}=\min\{u\}=-1$ \textit{and} $\max\{w\}=\max\{u\}={u^\star}$.

In order to prove that such a model function always exists, we need to study more properties of the one dimensional model. We begin with some definitions.
\begin{definition}
 Given the model function $w_{i,a}$, we define $b(i,a)$ to be the first value $b>a$ such that $\dot w_{i,a}(b)=0$, and set $b(a)=\infty$ if such a value doesn't exist. Equivalently, $b(i,a)$ is the first value $b>a$ such that $\phi_{i,a}(b)=\frac {\pi_p}{2}$. 
 
Define also the diameter of the model function as 
\begin{gather}
\delta(i,a)=b(i,a)-a 
\end{gather}
and the maximum of the model function
\begin{gather}
 m(i,a)= w_{i,a}(b(i,a))= \alpha^{-1} e_{i,a} (b(i,a))\, .
\end{gather}
\end{definition}
\begin{remark}\rm
 It is evident that, when $b(i,a)<\infty$, the range of $w$ on $[a,b]$ is $[-1,m]$. More precisely:
\begin{gather}
 w_{i,a}[a,b(i,a)]= [-1,m(i,a)]
\end{gather}
If $b(i,a)=\infty$, then $w_{i,a}[a,b(i,a))=[-1,m(i,a))$. In this case, we will see that $m(i,a)=0$.\\
An immediate consequence of Proposition \ref{prop_barasym} is that there always exists some $\bar a>0$ such that $b(3,-\bar a)<\infty$.
\end{remark}

In the following, we study the function $\delta(i,a)$, in particular its limit as $a$ goes to infinity. As we will see, the finiteness of this limit is related to the oscillatory behaviour of the differential equation. We will find a limiting value $\bar \alpha=\bar \alpha(k,n)$ such that 
for $\alpha>\bar \alpha$, $\delta(i,a)$ is finite for all $i,a$; while for $\alpha<\bar \alpha$, $\delta(i,a)=\infty$ for $i=1,2$ and $\lim_{a\to \infty} \delta(i,a)=\infty$. We begin by studying the translation invariant model $T=T_2$.

\begin{proposition}
 Consider the model $T_2$, then there exists $\bar \alpha(k,n)$ for which when $\alpha>\bar \alpha$ the solution to:
\begin{equation}
 \begin{cases}
  \dot \phi = \alpha + \frac{(n-1)\sqrt{-k}}{p-1} \cosp^{(p-1)}(\phi)\sinp(\phi)\\
  \phi(0)=-\frac{\pi_p}{2}
 \end{cases}
\end{equation}
has
\begin{gather}
 \lim_{t\to \infty} \phi(t)=\infty\, ,
\end{gather}
and in particular $\delta(2,0)<\infty$. While for $\alpha\leq \bar \alpha$:
\begin{gather}
-\frac{\pi_p} 2 < \lim_{t\to \infty} \phi(t)<0\, ,
\end{gather}
and $\delta(2,0)=\infty$.
\end{proposition}
\begin{proof}
 This problem is a sort of damped p-harmonic oscillator. The value $\bar \alpha$ is the critical value for the damping effect. The proof can be carried out in detail following the techniques used in the next Proposition, where $\bar\alpha$ is found explicitly in (\ref{eq_baralpha}).
\end{proof}

According to whether $\alpha>\bar \alpha$ or not, the behaviour of the solutions to the models $T_1$ and $T_3$ change in a similar fashion. We first describe what happens to the symmetric model, i.e., the model $T_3$.
\begin{proposition}\label{prop_alpha}
There exists a limiting value $\bar \alpha>0$ such that for $\alpha>\bar \alpha$ the solution $w_{3,a}$ has an oscillatory behaviour and $\delta(3,a)<\infty$ for every $a\in\R$.

For $\alpha<\bar \alpha$, instead, we have:
\begin{gather}
 \lim_{t\to \infty} \phi_{3,a}(t)<\infty
\end{gather}
for every $a\in \R$. Equivalently, for $a$ sufficiently large
\begin{gather}
 -\frac{\pi_p} 2 <\lim_{t\to \infty} \phi_{3,a}(t)<0
\end{gather}
and $\delta(3,a)=\infty$. For $\alpha = \bar \alpha$, we have:
\begin{gather}
 \lim_{a\to \infty} \delta(3,a) = \infty\, .
\end{gather}

\end{proposition}
\begin{proof}
We study the IVP:
\begin{gather}
 \begin{cases}
  \dot \phi = \alpha - \frac{T_3}{p-1} \cosp^{p-1}(\phi)\sinp(\phi)\\
  \phi(a)=-\frac{\pi_p}{2}
 \end{cases}
\end{gather}
We will only prove the claims on $\delta(3,a)$, and restrict ourselves to the case $a\geq -\bar a$. The other claims can be proved using a similar argument. 

Note that if there exists some $\bar t$ such that $\phi(\bar t)=0$, then necessarily $\bar t\geq 0$ and, for $s\in[\bar t, \phi^{-1}(\pi_p/2)]$, $\dot \phi(s)\geq \alpha$. So in this case $b-a<\bar t+\pi_p/(2\alpha)<\infty$.

Thus $b(3,a)$ (or equivalently $\delta(3,a)$) can be infinite only if $\phi<0$ indefinitely. Note that either $\dot \phi >0$ always, or $\dot \phi<0$ for all $t$ large. In fact, at those points where $\dot \phi=0$ we have:
\begin{gather}
 \ddot \phi = -\frac{\dot T_3}{T_3} \alpha\, .
\end{gather}
Since $\dot T_3<0$, and $T_3(t)<0$ for all $t>0$, once $\dot \phi$ is negative it can never turn positive again. So $\phi$ has always a limit at infinity, finite or otherwise.

By simple considerations on the ODE, if $b(3,a)=\infty$, $\lim_{t\to \infty} \phi(t)$ can only be a solution $\psi$ of:
\begin{gather}\label{eq_psi}
 0=\alpha- \frac{-(n-1)\sqrt{-k} }{p-1} \cosp^{(p-1)}(\psi)\sinp(\psi)\equiv F(\psi)\, .
\end{gather}
Since $\alpha>0$, it is evident that this equation does not have solutions in $[0,\pi_p/2]$. Studying the function $\cosp^{(p-1)}(\psi)\sinp(\psi)$ on $\ton{-\pi_p/2,0}$, we notice that this function is negative and has a single minimum $-l$ on this interval.  Now set
\begin{gather}\label{eq_baralpha}
 \bar \alpha = \frac{(n-1)l\sqrt{-k}}{(p-1)}\, ,
\end{gather}
so that for $\alpha >\bar \alpha$, $F(\psi)$ has a positive minimum, for $\alpha=\bar \alpha$, its minimum is zero, and for $0<\alpha <\bar \alpha$, its minimum is negative.

\paragraph{\textbf{Case 1: $\alpha=\bar \alpha$}}
Before turning to the model $T_3$, in this case we briefly discuss what happens in the model $T_2$ \footnote{recall that this model is translation invariant, so we only need to study the solution $\phi_{2,0}=\phi_2$}, in particular we study the function:
\begin{gather}
 \begin{cases}
  \dot \phi_{2}= \alpha + \frac{(n-1)\sqrt{-k}}{p-1}\cosp^{(p-1)}(\phi_2)\sinp(\phi_2)=F(\phi_2)\\
  \phi_2(0)=-\frac{\pi_p}2\, .
 \end{cases}
\end{gather}
Since $\alpha=\bar \alpha$, it is easy to see that $\dot \phi_2\geq0$ everywhere. Let $\psi\in\ton{-\pi_p/2,0}$ be the only solution of $F(\psi)=0$. Since $\psi$ satisfies the differential equation $\dot \psi = F(\psi)$, by uniqueness we have that $\phi_2\leq \psi$ everywhere, and thus the function $\phi_2$ is strictly increasing and has a finite limit at infinity:
\begin{gather}
 \lim_{t\to \infty} \phi_2(t)=\psi\, .
\end{gather}
With this information in mind, we turn our attention back to the model $T_3$. In some sense, the bigger $a$ is, the closer the function $T_3(a+t)$ is to the constant function $T_2$. Consider in fact the solution $\phi_{3,a}(t)$, and for convenience translate the independent variable by $t\to t-a$. The function $\tau \phi_{3,a} (t) = \phi_{3,a} (t+a)$ solves:
\begin{gather}
  \begin{cases}
  \tau \dot \phi_{3,a}= \alpha- \frac{T_3(a+t)}{p-1}\cosp^{(p-1)}(\phi)\sinp(\phi)\\
  \tau \phi_{3,a}(0)=-\frac{-\pi_p}2
 \end{cases}
\end{gather}
Since $T_3(a+t)$ converges in $C^1([0,\infty))$ to $T_2=-(n-1)\sqrt{-k}$, we have that
\begin{gather}
 \lim_{a\to \infty} \tau \phi_{3,a} = \phi_2
\end{gather}
in the sense of local $C^1$ convergence on $[0,\infty)$. This implies immediately that:
\begin{gather}
\lim_{a\to \infty} \delta(3,a)= \lim_{a\to \infty} b(3,a)-a =b(2,0)=\infty\, .
\end{gather}

\paragraph{\textbf{Case 2: }}
if $0<\alpha <\bar \alpha$, there are two solutions $-\pi_p/2<\psi_1< \psi_2 <0$ to equation \eqref{eq_psi}. Take $\epsilon>0$ small enough such that $\psi_2-\epsilon>\psi_1$. Thus there exists $\epsilon'>0$ such that
\begin{gather}
 \frac{d}{dt}(\psi_2-\epsilon)=0>\alpha+ \frac{(n-1)\sqrt{-k}-\epsilon'}{(p-1)}\cosp^{(p-1)}(\psi_2-\epsilon)\sinp(\psi_2-\epsilon)>\\
\notag >\alpha+ \frac{(n-1)\sqrt{-k}}{(p-1)}\cosp^{(p-1)}(\psi_2-\epsilon)\sinp(\psi_2-\epsilon)\, .
\end{gather}
Since $\lim_{t\to \infty} T_3(t)=-(n-1)\sqrt{-k}$, there exists an $A>>1$ such that $T_3(t)\leq -(n-1)\sqrt{-k}+\epsilon'$ for $t\geq A$. 

Choose $a>A$, and consider that, as long as $\phi_{3,a}(t)<0$:
\begin{gather}
 \begin{cases}
  \dot \phi_{3,a}= \alpha- \frac{T}{p-1}\cosp^{(p-1)}(\phi)\sinp(\phi)\leq \alpha + \frac{-\epsilon'+(n-1)\sqrt{-k} }{p-1} \cosp^{(p-1)}(\phi)\sinp(\phi)\\
  \phi_{3,a}(a)=-\frac{-\pi_p}2
 \end{cases}
\end{gather}
Then, by a standard comparison theorem for ODE, $\phi_{3,a}\leq \psi_2-\epsilon$ always.

In particular, $\lim_{t\to \infty}\phi_{3,a}(t)=\psi_1$, and using equation \eqref{eq_pf_nege} we also have
\begin{gather}\label{eq_0lim}
 \lim_{t\to \infty} e_{3,a}= \lim_{t\to \infty} w_{3,a}=0\, .
\end{gather}
It is also evident that, if $a>A$, $\delta(3,a)=\infty$.

\paragraph{\textbf{Case 3}}

If $\alpha > \bar \alpha$, then there exists a positive $\epsilon$ such that $\dot \phi_{i,a}\geq \epsilon$ for $i=2,3$ and all $a\in \R$. Thus with a simple estimate we obtain for $i=2,3$:
\begin{gather}
 \delta(i,a)\leq \frac{\pi_p}{\epsilon} \, .
\end{gather}
Moreover, as in case $1$, it is easy to see that $\phi_{3,a}(t-a)$ converges locally uniformly in $C^1$ to $\phi_{2,0}$, and so, in particular:
\begin{gather}
 \lim_{a\to \infty} \delta(3,a) = b(2,0)\, .
\end{gather}
\end{proof}

As for the model $T_1$, an analogous argument leads to the following Proposition:
\begin{proposition}
 Consider the model $T_1$. If $\alpha>\bar \alpha$, then the solutions have an oscillatory behavior with $\phi_{1,a}(\infty)=\infty$ and $\delta(1,a)<\infty$ for all $a\in [0,\infty)$. If $\alpha\leq \bar \alpha$, then $\phi_{1,a}$ has a finite limit at infinity and $\delta(1,a)=\infty$ for all $a\in [0,\infty)$. 
\end{proposition}

Now we turn our attention to the maximum $m(i,a)$ of the model functions $w_{i,a}$. Our objective is to show that for every possible $0<{u^\star}\leq 1$, there exists a model such that $m(i,a)={u^\star}$. This is immediately seen to be true if $\alpha\leq \bar \alpha$. Indeed, in this case we have:
\begin{proposition}\label{prop_kasmall}
 Let $\alpha\leq \bar \alpha$. Then for each $0<{u^\star}\leq 1$, there exists an $a\in[-\bar a, \infty)$ such that $m(3,a)={u^\star}$.
\end{proposition}
\begin{proof}
 Proposition \ref{prop_barasym} shows that this is true for ${u^\star}=1$. For the other values, we know that if $\alpha\leq \bar \alpha$:
\begin{gather}
 \lim_{a\to \infty}\delta(3,a)=\infty\, .
\end{gather}
By equation \eqref{eq_0lim} (or a similar argument for $\alpha=\bar \alpha$) and using the continuity with respect to the parameters of the solution of our ODE, it is easy to see that:
\begin{gather}
 \lim_{a\to a^\star} m(3,a)= 0 \, ,
\end{gather}
where $a^\star $ is the first value for which $\delta(3,a^\star)=\infty$ (which may be infinite if $\alpha=\bar \alpha$).

Since $m(3,a)$ is a continuous function and $m(3,-\bar a)=1$, we have proved the Proposition.
\end{proof}

The case $\alpha>\bar \alpha$ requires more attention. First of all, we prove that the function $m(i,a)$ is invertible.
\begin{proposition}
If $m(i,a)=m(i,s)>0$, then $w_{i,a}$ is a translation of $w_{i,s}$. In particular, if $i\neq 2$, $a=s$.
\end{proposition}
\begin{proof}
 Note that if $i=2$, the model is translation invariant and the proposition is trivially true. In the other cases, the proof follows from an application of Theorem \ref{grad_est_negw}. Since $m(i,a)=m(i,s)>0$, we know that $b(i,a)$ and $b(i,s)$ are both finite. So our hypothesis imply that:
\begin{gather}
 \dot w_{i,a}|_{w_{i,a}^{-1}} = \dot w_{i,s}|_{w_{i,s}^{-1}}\, .
\end{gather}
By the uniqueness of the solutions of the IVP \eqref{eq_1dm_neg}, we have that $w_{i,a}(t)=w_{i,s}(t+t_0)$, which, if $i\neq 2$, is possible only if $a=s$.
\end{proof}

If $\alpha>\bar \alpha$, then $m(2,a)$ is well-defined, positive, strictly smaller than $1$ and independent of $a$. We define $m_2=m(2,a)$.
\begin{proposition}\label{p_1}
 If $\alpha>\bar \alpha$, then $m(3,a)$ is a decreasing function of $a$, and:
\begin{gather}
 \lim_{a\to \infty} m(3,a)= m_2\, .
\end{gather}
\end{proposition}
\begin{proof}
 This proposition is an easy consequence of the convergence property described in the proof of Proposition \ref{prop_alpha}. Since $m(3,a)$ is continuous, well defined on the whole real line and invertible, it has to be decreasing.
\end{proof}

We have just proved that, for $a\to \infty$, $m(3,a)$ decreases to $m_2$. With a similar technique, we can show that, for $a\to \infty$, $m_{1,a}$ increases to $m_2$.
\begin{proposition}\label{p_2}
 If $\alpha>\bar \alpha$, then for all $a\geq 0$, $b(1,a)<\infty$. Moreover, $m(1,a)$ is an increasing function on $[0,\infty)$ such that:
\begin{gather}
 m(1,0)=m_0>0 \quad \text{and} \quad \lim_{a\to \infty} m(1,a)=m_2\, .
\end{gather}
\end{proposition}

Using the continuity of $m(i,a)$ with respect to $a$, we get as a corollary the following proposition.
\begin{proposition}\label{p3}
 For $\alpha>\bar \alpha$ and for any $u^\star \in [m(1,0),1]$, there exists some $a\in \R$ and $i\in \{1,2,3\}$ such that $m(i,a)=u^\star$.
\end{proposition}

In the next section we address the following question: is it possible that $u^\star<m(1,0)$? Using a volume comparison theorem, we will see that the answer is no. Thus there always exists a model function $w_{i,a}$ that fits perfectly the eigenfunction $u$.

\section{Maxima of eigenfunctions and volume comparison}\label{sec_max}
In order to prove that (if $\alpha >\bar \alpha$) $u^\star \geq m(1,0)$, we adapt a volume comparison technique which was introduced in \cite[Section 6]{new}. For the sake of completeness, in this section we carry out all the proofs in details, however the theorems proved here are very similar to the ones proved in \cite[Section 6]{svelto}. We start with the definition of the metric $dm$ obtained as the pull-back of the Riemannian volume.
\begin{definition}
 Given the eigenfunction $u$ and $w$ as in Theorem \ref{grad_est_neg}, let $t_0\in (a,b)$ be the unique zero of $w$ and let $g\equiv w^{-1}\circ u$. We define the measure $m$ on $[a,b]$ by
\begin{gather*}
 m(A)\equiv \V(g^{-1}(A)) \ ,
\end{gather*}
where $\V$ is the Riemannian measure on $M$. Equivalently, for any bounded measurable $f:[a,b]\to \R$, we have
\begin{gather*}
 \int_a^b f(s) dm(s) = \int_M f(g(x))d\V(x) \ . 
\end{gather*}

\end{definition}
\begin{theorem}\label{teo_comp}
 Let $u$ and $w$ be as above, and let
\begin{gather*}
 E(s)\equiv -\operatorname{exp}\ton{\lambda\int_{t_0}^s \frac{w^{(p-1)}}{\dot w^{(p-1)}} dt}\int_a ^s{w(r)}^{(p-1)} dm(r)
\end{gather*}
Then $E(s)$ is increasing on $(a,t_0]$ and decreasing on $[t_0,b)$.
\end{theorem}
Before the proof, we note that this theorem can be rewritten in a more convenient way. Consider in fact that by definition
\begin{gather*}
 \int_{a}^s \pw(r)\ dm(r)= \int_{\{u\leq w(s) \}}u(x)^{(p-1)}\ d\V(x) \ .
\end{gather*}
Moreover, note that the function $w$ satisfies
\begin{gather*}
 \frac d {dt}(\mu_i \pdw)=-\lambda \mu_i \pw \ , \\
 -\lambda\frac{\pw}{\pdw}=
\frac d {dt} \log(\mu_i \pdw) \ ,
\end{gather*}
and therefore
\begin{gather*}
 -\lambda\int_a ^s \pw(t)\mu_i (t) dt = \mu_i(s) \pdw (s) \ , \\
\operatorname{exp}\ton{\lambda \int_{t_0}^s \frac{\pw}{\pdw} dt }=\frac{\mu_i(t_0)\pdw(t_0)}{\mu_i(s) \pdw(s)} \ .
\end{gather*}
Thus, the function $E(s)$ can be rewritten as
\begin{gather*}
 E(s)=C\frac{\int_a ^s\pw(t) \ dm(r)}{\int_a ^s \pw(t)\ \mu_i(t) dt}=C\frac{\int_{\{u\leq w(s)\}}{u(x)}^{(p-1)}\ d\V(x)}{\int_a ^s {w(t)}^{(p-1)}\ \mu_i(t) dt} \ ,
\end{gather*}
where $\lambda C^{-1}=\mu_i(t_0)\pdw(t_0)$, and the previous theorem can be restated as follows.
\begin{theorem}\label{int_comp}
 Under the hypothesis of the previous theorem, the ratio
\begin{gather*}
 E(s)=\frac{\int_a ^s\pw(r) \ dm(r)}{\int_a ^s \pw(t)\ \mu_i(t) dt}=\frac{\int_{\{u\leq w(s)\}}{u(x)}^{(p-1)}\ d\V(x)}{\int_a ^s {w(t)}^{(p-1)}\mu_i(t) dt}
\end{gather*}
is increasing on $[a,t_0]$ and decreasing on $[t_0,b]$.
\end{theorem}
\begin{proof}[\textsc{Proof of Theorem \ref{teo_comp}}]
 Chose any smooth nonnegative function $H(s)$ with compact support in $(a,b)$, and define $G:[-1,w(b)]\to \R$ in such a way that
\begin{gather*}
 \frac d {dt} \qua{{G(w(t))}^{(p-1)}} =H(t) \quad \ \ G(-1)=0 \ .
\end{gather*}
It follows that
\begin{gather*}
 \pG(w(t))= \int_a ^t H(s) ds \ \ \ \ (p-1) \abs{G(w(t))}^{p-2} \dot G(w(t)) \dot w (t) = H(t) \ .
\end{gather*}
Then choose a function $K$ such that $(t K(t))'=K(t)+t \dot K(t)=G(t)$. By the chain rule we obtain
\begin{gather*}
 \Delta_p (uK(u))=G^{(p-1)} (u) \Delta_p(u) + (p-1)\abs{G(u)}^{p-2} \dot G(u) \abs{\nabla u}^p \ .
\end{gather*}
Using the weak formulation of the divergence theorem, it is straightforward to verify that
\begin{gather*}
 \int_M \Delta_p(uK(u))d \V =0\, ,
\end{gather*}
and so we get
\begin{gather*}
\frac{\lambda  }{p-1} \int_M \pu \pG(u) d\V(x)= \int_M \abs{G(u)}^{p-2} \dot G (u) \abs{\nabla u}^{p} d\V \ .
\end{gather*}
Applying the gradient comparison Theorem \ref{grad_est_neg}, noting that we consider only $\lambda>0$, we have
\begin{gather*}
\frac{\lambda  }{p-1} \int_M \pu \pG (u)d\V(x)\leq \int_M \abs{G(u)}^{p-2} \dot G (u) (\dot w \circ w^{-1}(u)) ^p d\V \ .
\end{gather*}
By definition of $d m $, the last inequality can be written as
\begin{gather*}
\frac{\lambda  }{p-1} \int_a^b \pw(s) \pG(w(s)) dm(s)\leq \\
\leq\int_a^b \abs{G(w(s))}^{p-2} \dot G (w(s)) (\dot w(s)) ^p dm(s) \ ,
\end{gather*}
and recalling the definition of $G$ we deduce that
\begin{gather*}
 \lambda \int_a^b \pw(s) \ton{\int_a^s H(t) dt}dm(s)= \lambda \int_a^b  \ton{\int_s^b \pw(t) dm(t)} H(s)ds \leq \\
\leq \int_a^b H(s) \pdw(s)\ dm(s) \ .
\end{gather*}
Since $\int_a^b \pw(t) dm(t)=0$, we can rewrite the last inequality as
\begin{gather*}
 \int_a^b  H(s) \qua{-\lambda\int_a^s \pw(t) dm(t)} ds \leq \int_a^b H(s) \pdw(s)\ dm(s) \ .
\end{gather*}
Define the function $A(s)\equiv -\int_a^s \pw(r) dm(r)$. Since the last inequality is valid for all smooth nonnegative function $H$ with compact support, then
\begin{gather*}
\pdw(s) dm(s)-\lambda A(s)ds \geq 0 
\end{gather*}
in the sense of distributions, and therefore the left hand side is a positive measure. In other words, the measure $\lambda Ads+ \frac{\pdw}{\pw} dA$ is nonpositive. Of if we multiply the last inequality by $\frac{\pw}{\pdw}$, and recall that $w\geq 0$ on $[t_0,b)$ and $w\leq 0$ on $(a,t_0]$, we conclude that the measure 
\begin{gather*}
 \lambda  \frac{\pw}{\pdw}Ads+  dA
\end{gather*}
is nonnegative on $(a,t_0]$ and nonpositive on $[t_0,b)$, or equivalently the function
\begin{gather*}
 E(s)= A(s) \operatorname{exp}\ton{\lambda\int_{t_0}^s \frac{\pw}{\pdw}(r) dr}
\end{gather*}
is increasing on $(a,t_0]$ and decreasing on $[t_0,a)$.
\end{proof}

Before we state the comparison principle for maxima of eigenfunctions, we need the following lemma. The definitions are consistent with the ones in Theorem \ref{grad_est_neg}.
\begin{lemma}\label{lemma_radepsilon}
 For $\epsilon$ sufficiently small, the set $u^{-1}[-1,-1+\epsilon)$ contains a ball of radius $r=r_\epsilon$, which is determined by
\begin{gather*}
r_\epsilon=w ^{-1}(-1+\epsilon)-a \ .
\end{gather*}
\begin{proof}
 This is a simple application of the gradient comparison Theorem \ref{grad_est_neg}. Let $x_0$ be a minimum point of $u$, i.e. $u(x_0)=-1$, and let $\bar x$ be another point in the manifold. Let $\gamma:[0,l]\to M$ be a unit speed minimizing geodesic from $x_0$ to $\bar x$, and define $f(t)\equiv u(\gamma(t))$. It is easy to see that
\begin{gather}\label{ref_h}
\abs{\dot f (t)}=\abs{\ps{\nabla u |_{\gamma(t)}}{\dot \gamma(t)}}\leq \abs{\nabla u |_{\gamma(t)}}\leq \dot w |_{w^{-1}(f(t))}  \ .
\end{gather}
Since
\begin{gather*}
\frac d {dt} w^{-1}(f(t))\leq 1 \ ,
\end{gather*}
we have that $a\leq w^{-1}(f(t))\leq a+t$, and since $\dot w$ is increasing in a neighborhood of $a$, we can deduce that
\[\dot w |_{w^{-1}f(t)}  \leq \dot w |_{a+t} \ .\]
By the absolute continuity of $u$ and $\gamma$, we can conclude that
\begin{gather*}
 \abs{f(t)+1}\leq\int_0^t  \dot w |_{a+s} ds = (w(a+ t)+1) \ .
\end{gather*}
This means that if $l=d(x_0,\bar x)< w ^{-1}(-1+\epsilon)-a$, then $u(\bar x)< -1+\epsilon$.
\end{proof}

\end{lemma}

And now we are ready to prove the comparison theorem. 
\begin{theorem}
If $u$ is an eigenfunction on $M$ such that $\min\{u\}=-1=u(x_0)$ and $\max\{u\}\leq m(k,n,1,0)$, then for every $r>0$ sufficiently small, the volume of the ball centered at $x_0$ and of radius $r$ is controlled by
\begin{gather*}
 \V(B(x_0,r))\leq c r^n \ .
\end{gather*}

\end{theorem}
\begin{proof}
 For simplicity, fix $k,n$, $i=1$ and $a=0$, and denote $w=w_{k,n,1,0}$, $m=m(k,n,1,0)$. Define also the measure $d\nu$ by $d\nu(t)= \mu_1(t) dt$.\\
 For $k\leq -1/2^{p-1}$, applying Theorem \ref{int_comp} we can estimate
\begin{gather*}
 \V(\{u\leq k\})\leq -2 \int_{\{u\leq k\}}  u^{(p-1)} d\V \leq\\
\leq -2 C \int_{\{w\leq k\}} w^{(p-1)} d\nu\leq 2C\nu(\{w\leq k\}) \ .
\end{gather*}
If we set $k=-1+\epsilon$ for $\epsilon$ small enough, it follows from Lemma \ref{lemma_radepsilon} that there exist positive constants $C$ and $C'$ such that
\begin{gather*}
\V(B(x_0,r_\epsilon))\leq  \V(\{u\leq k\})\leq \\
\leq2C\nu(\{w\leq -1+\epsilon\})= 2C \nu([0,r_\epsilon])= C' \int_0^{r_\epsilon} \sinh^{n-1}(t) dt \leq 2 C' r_\epsilon ^n \ .
\end{gather*}
\end{proof}

As an immediate corollary, we get the following proposition, which answers to the question raised at the end of the previous section.
\begin{proposition}
 Let $u:M\to \R$ be an eigenfunction of the $p$-Laplacian, and suppose that $\alpha>\bar \alpha$. Then $\max\{u\}=u^\star\geq m(k,n,1,0)>0$.
\end{proposition}
\begin{proof}
 Suppose by contradiction that $\max\{u\}<m(k,n,i,0)$. Then, by the continuous dependence of solutions of ODE \eqref{eq_1dm_neg} on the parameters, there exists $n'>n$ ($n'\in \R$) such that $\max\{u\}\leq m(k,n',i,0)$. Note that, since Corollary \ref{cor_n} is valid for all $n'\geq n$, we can still apply the gradient comparison theorem to get
 \begin{gather*}
  \abs{\nabla u} (x)\leq \dot w |_{w^{-1}(u(x))}
 \end{gather*}
where $w=w_{k,n',1,0}$. Thus also the volume comparison remains valid, but this implies for small $\epsilon$ (which means for $r_\epsilon$ small)  $\V(B(x_0,r_\epsilon))\leq c r_\epsilon^{n'}$, which contradicts the assumption that $M$ is $n$ dimensional. Note that the argument applies even in the case where $M$ has a $C^2$ boundary.
\end{proof}

Finally, as a corollary of this Proposition and Proposition \ref{p3}, we get the following.
\begin{corollary}\label{cor_maxfit}
 Let $u:M\to \R$ be an eigenfunction of the $p$-Laplacian. Then there always exists $i\in \{1,2,3\}$ and $a\in I_i$ such that $u^\star=\max\{u\}=m(i,a)$. This means that there always exists a solution of equation \eqref{eq_1dm_neg} relative to the model $1,2$ or $3$ such that:
\begin{gather}
 u(M)= [-1,w(b)]\, .
\end{gather}
\end{corollary}

\section{Diameter comparison}\label{s:diamter}
In this section we study the diameter $\delta(i,a)=b(i,a)-a$ as a function of $i$ and $a$, having fixed $n, \, k$ and $\lambda$. In particular, we are interested in characterizing the minimum possible value for the diameter.

\begin{definition}\label{deph_delta}
 For fixed $n, \, k$ and $\lambda$, define $\bar \delta$ by:
\begin{gather}
 \bar \delta(n,k,\lambda) =\bar \delta =\min \{\delta(i,a), \ \ i=1,2,3,\ a\in I_i\}
\end{gather}

\end{definition}

By an application of Jensen's inequality, we obtain a simple lower bound on $\delta(i,a)$ for $i=1,2$.
\begin{proposition}
 For $i=1,2$ and for any $a\in I_i$, $\delta(i,a)>\frac{\pi_p}{\alpha}$.
\end{proposition}
\begin{proof}
We can rephrase the estimate in the following way: consider the solution $\phi(i,a)$ of the initial value problem
\begin{gather*}
\begin{cases}
  \dot \phi = \alpha - \frac{T_i}{(p-1)}\cosp^{p-1} (\phi)\sinp(\phi)\\
\phi(a)=-\frac{\pi_p}{2} 
\end{cases}
\end{gather*}
Then $b(i,a)$ is the first value $b>a$ such that $\phi(i,b)=\frac{\pi_p}2$, and $\delta(i,a)=b(i,a)-a$.

We start by studying the translation invariant model $T_2$. In this case, using separation of variables, we can find the solution $\phi(2,0)$ in an implicit form. Indeed, if $\alpha\leq \bar \alpha$, then we have already shown in Proposition \ref{prop_alpha} that $\delta(i,a)=\infty$. If $\alpha>\bar \alpha$, we have $\dot \phi >0$ and:
\begin{gather*}
 \delta(2,0)=b(2,0)-0=\int_{-\frac {\pi_p} 2}^{\frac {\pi_p} 2} \frac{d\psi}{\alpha + \gamma\cosp^{p-1} (\psi)\sinp(\psi)} \, ,
\end{gather*}
where $\gamma=-\frac{T_2}{p-1}$ is a nonzero constant. Since $\cosp^{p-1}(\psi)\sinp(\psi)$ is an odd function, by Jensen's inequality we can estimate
\begin{gather}\label{eq_jan_neg}
\frac{\delta(2,0)}{\pi_p}= \frac 1 {\pi_p} \int_{-\pi_p/2}^{\pi_p/2} \frac{d\psi}{\alpha+\gamma \cosp^{p-1} (\psi)\sinp(\psi)}>\\
>\qua{\frac 1 {\pi_p}\int_{-\pi_p/2}^{\pi_p/2} \ton{\alpha+\gamma \cosp^{p-1} (\psi)\sinp(\psi)}d\psi}^{-1} = \frac{1}{\alpha} \ . 
\end{gather}
Note that, since $T_2\neq 0$, this inequality is strict.

If $i=1$ and $\alpha\leq \bar \alpha$, we still have $\delta(1,a)=\infty$ $\forall a\geq 0$. On the other hand, if $\alpha >\bar \alpha$ we can use the fact that $\dot T_1>0$ to compare the solution $\phi(1,a)$ with a function easier to study. Let $t_0$ be the only value of time for which $\phi(1,a)(t_0)=0$. Then it is easily seen that:
\begin{gather}
 \dot \phi(1,a)= \alpha - \frac{T_1}{p-1} \cosp^{(p-1)}(\phi)\sinp(\phi) \leq \alpha-\frac{T_1(t_0)}{p-1} \cosp^{(p-1)}(\phi)\sinp(\phi)\, .
\end{gather}
Define $\gamma=\frac{T_1(t_0)}{p-1}$. Using a standard comparison theorem for ODE, we know that, for $t>a$
\begin{gather}
 \phi(1,a)(t)< \psi(t)\, ,
\end{gather}
where $\psi$ is the solution to the IVP
\begin{gather}
 \begin{cases}
  \dot \psi = \alpha-\gamma \cosp^{(p-1)}(\psi)\sinp(\psi)\\
  \psi(a)=-\frac{\pi_p}2
 \end{cases}\, .
\end{gather}
If we define $c(a)$ to be the first value of time $c>a$ such that $\psi(c)=\pi_p/2$, then we have $b(2,a)\geq c(a)$. Using separation of variables and Jensen's inequality as above, it is easy to conclude that:
\begin{gather}
 \delta(1,a)>c(a)-a>\frac{\pi_p}{\alpha}\, .
\end{gather}
\end{proof}

\begin{remark}
 \rm For the odd solution $\phi_{3,-\bar a}$, it is easy to see that $\dot \phi \geq \alpha$ on $[-\bar a,\bar a]$ with strict inequality on $(-\bar a, 0)\cup(0,\bar a)$. For this reason:
 \begin{gather*}
  \delta(3,-\bar a)<\frac{\pi_p}{\alpha}\, ,
 \end{gather*}
and so $\bar \delta$ is attained for $i=3$.
\end{remark}

In the following proposition we prove that $\bar \delta = \delta(3,-\bar a)$, and for all $a\neq \bar a$ the strict inequality $\delta(3,a)>\bar \delta$ holds.

\begin{proposition}
 For all $a\in I_3=\R$:
 \begin{gather*}
  \delta(3,a)\geq \delta (3,-\bar a) = 2\bar a = \bar \delta\, ,
 \end{gather*}
with strict inequality if $a\neq -\bar a$.
\end{proposition}
\begin{proof}
 The proof is based on the symmetries and the convexity properties of the function $T_3$. Fix any $a>-\bar a$ (with an analogue argument it is possible to deal with the case $a<-\bar a$), and set
 \begin{gather*}
 \psi_+(t) = \phi_{3,a}(t) \, ,\quad \quad  \varphi (t)= \phi_{3,-\bar a}(t) \, ,\quad \quad \psi_-(t) = -\psi_+(-t)\, .
 \end{gather*}
We study these functions only when their range is in $[-\pi_p/2,\pi_p/2]$, and since we can assume that $b(3,a)<\infty$, we know that $\dot \psi_{\pm} >0$ on this set (see the proof of Proposition \ref{prop_alpha}). Using the symmetries of the IVP \eqref{eq_pf_neg}, it is easily seen that the function $\varphi$ is an odd function and that $\psi_-$ is still a solution to \eqref{eq_pf_neg}. In particular:
\begin{gather*}
 \psi_- (t) = \phi_{3,-b(3,a)}(t)\, .
\end{gather*}
Note that by comparison, we always have $\psi_-(t) > \varphi(t) > \psi_+(t)$.

Since all functions have positive derivative, we can study their inverses:
\begin{gather*}
 h= \psi_-^{-1}\, , \quad \quad s = \varphi^{-1}\, , \quad \quad g = \psi_+ ^{-1}\, .
\end{gather*}

Set for simplicity
\begin{gather*}
 f(\phi)\equiv \frac 1 {p-1} \cosp^{(p-1)}(\phi)\sinp(\phi)
\end{gather*}
and note that on $[-\pi_p/2,\pi_p/2]$, $f(\phi)$ is odd and has the same sign as $\phi$. 
The function defined by:
\begin{gather*}
 m(\phi) = \frac{1}{2} \ton{h(\phi) + g(\phi)}
\end{gather*}
is an odd function such that $m(0)=0$ and 
\begin{gather*}
 m(\pi_p/2) = \frac{1}{2} \ton{h(\pi_p/2)+g(\pi_p/2)} = \frac {1}{2} \ton{b(3,a) -a  } = \frac {1}{2} \delta (3,a)\, ,
\end{gather*}
thus the claim of the proposition is equivalent to $m(\pi_p/2) >\bar a$.

By symmetry, we restrict our study to the set $\phi \geq 0$, or equivalently $m \geq 0$. Note that $m$ satisfies the following ODE:
\begin{gather*}
 2 \frac{dm}{d\phi} =  \frac{1}{\alpha -  T_3(g) f(\phi)} +  \frac{1}{\alpha -  T_3(h) f(\phi)} \, .
\end{gather*}
Fix some $\alpha,\beta\in \R^+ $ and consider the function:
\begin{gather}
 z(t) = \frac{1}{\alpha - \beta T_3(t)}\, .
\end{gather}
Its second derivative is:
\begin{gather*}
 \ddot z = \frac{2\beta^2 \dot T_3^2}{(\alpha- \beta T_3)^3 } + \frac{\beta \ddot T_3}{(\alpha- \beta T_3)^2}\, .
\end{gather*}
So, if $\beta\geq 0$ and $t\geq 0$, $z$ is a convex function. In particular this implies that:
\begin{gather}\label{eq_conv}
 \frac{dm}{d\phi} \geq \frac{1}{\alpha - T_3(m) f(\phi)} 
\end{gather}
for all those values of $\phi$ such that both $g$ and $h$ are nonnegative. However, by symmetry, it is easily seen that this inequality holds also when one of the two is negative. Indeed, if $h<0$, we have that:
\begin{gather*}
 \frac 1 2 \qua{\frac{1}{\alpha-\beta T_3(h)} + \frac{1}{\alpha-\beta T_3(-h)}} \geq \frac 1 \alpha  = \frac{1}{\alpha-\beta T_3\ton{\frac{h-h}{2}}}\, .
\end{gather*}
Recall that if for for some $-t<t<c$, $f(0)\leq [f(-t)+f(t)]/2$ and $f$ is convex on $[0,c]$, then $f((-t+c)/2) \leq [f(-t)+f(c)]/2$, so inequality \eqref{eq_conv} follows. Moreover, note that if $\beta > 0$ (i.e. if $\phi \in (0,\pi_p)$) and if $g\neq h$, the inequality is strict.

Using a standard comparison for ODE, we conclude that $m(\phi)\geq s(\phi)$ on $[0,\pi_p/2]$ and in particular:
\begin{gather}
m\ton{\pi_p/2} > s\ton{\pi_p/2} =\bar a\, ,
\end{gather}
and the claim follows immediately.
\end{proof}

In the last part of this section, we study $\bar \delta$ as a function of $\lambda$, having fixed $n$ and $k$. Given the previous proposition, it is easily seen that $\bar \delta(\lambda)$ is a strictly decreasing function, and so invertible. In particular, we can define its inverse $\bar \lambda (\delta)$, and characterize it in the following equivalent ways (see also Definition \ref{deph_l} and Remark \ref{rem_l}).
\begin{proposition}\label{prop_barlambda}
 For fixed $n$, $k$ and $p>1$, we have that given $\delta>0$ that $\bar \lambda(n,k,\delta)$ is the first positive Neumann eigenvalue on $[-\delta/2,\delta/2]$ relative to the operator
 \begin{gather*}
  \frac{d}{dt}\ton{\dot w^{(p-1)}} + (n-1)\sqrt{-k}\tanh\ton{\sqrt{-k} t} \dot w ^{(p-1)}+\lambda w^{(p-1)}\, ,
 \end{gather*}
or equivalently $\bar \lambda$ is the unique value of $\lambda$ such that the solution to
\begin{gather*}
 \begin{cases}
  \dot \phi = \ton{\frac{\lambda}{p-1}}^{1/p} + \frac{(n-1)\sqrt{-k}}{p-1}\tanh\ton{\sqrt{-k} t}  \cosp^{(p-1)}(\phi) \sinp(\phi)\\
  \phi(0)=0
 \end{cases}
\end{gather*}
satisfies $\phi(\delta/2) = \pi_p/2$.
\end{proposition}
\begin{remark}
 \rm It is easily seen that the function $\bar \lambda (n,k,\delta)$ is a continuous function with respect to its {pa\-ra\-me\-ters}, moreover it has the following monotonicity properties:
 \begin{gather*}
   \delta_1 \leq \delta_2 \ \ \ \text{and} \ \ \ n_1\geq n_2 \ \ \ \text{and} \ \ \ k_1\geq k_2 \Longrightarrow \bar \lambda (n_1,k_1,\delta_1)\geq \bar \lambda (n_2,k_2,\delta_2)\, .
 \end{gather*}

\end{remark}

\section{Sharp estimate}\label{sec_main}
Now we are ready to state and prove the main Theorem on the spectral gap.
\begin{theorem}\label{teo_main_proof}
 Let $M$ be a compact $n$-dimensional Riemannian manifold with Ricci curvature bounded from below by $(n-1)k<0$, diameter $d$ and possibly with convex $C^2$ boundary. Let $\lambda_{1,p}$ be the first positive eigenvalue of the $p$-Laplacian (with Neumann boundary condition if necessary). Then 
\begin{gather*}
 \lambda_{1,p}\geq \bar \lambda(n,k,d) \ ,
\end{gather*}
where $\bar \lambda$ is the function defined in Proposition \ref{prop_barlambda}.
\end{theorem}
\begin{proof}
 To begin with, we rescale $u$ in such a way that $\min\{u\}=-1$ and $0<u^\star=\max\{u\}\leq 1$. By Corollary \ref{cor_maxfit}, we can find a solution $w^{p,\lambda}_{k,n,i,a}$ such that $\max\{u\}=\max\{w \ \text{ on }\ [a,b(a)]\}=m(k,n,i,a)$.

Consider a minimum point $x$ and a maximum point $y$ for the function $u$, and consider a unit speed minimizing geodesic (of length $l\leq d$) joining $x$ and $y$. Let $f(t)\equiv u(\gamma(t))$, and choose some $I\subseteq [0,l]$ in such a way that $I\subseteq \dot f^{-1} (0,\infty)$ and $f^{-1}$ is well defined in a subset of full measure of $[-1,u^\star]$. Then, by a simple change of variables and an easy application of the gradient comparison Theorem \ref{grad_est_neg}, we get
 \begin{gather*}
  d\geq \int_0 ^l dt \geq \int_I dt \geq \int_{-1}^{u^\star} \frac{dy}{\dot f (f^{-1}(y))}\geq\int_{-1}^{u^\star} \frac{dy}{\dot w (w^{-1}(y))}=\\
=\int_a^{b(a)} 1 dt = \delta(k,n,i,a)\geq \bar \delta(n,k,\lambda)\ ,
 \end{gather*}
where the last inequality follows directly from the Definition \ref{deph_delta}. This and Proposition \ref{prop_barlambda} yield immediately to the estimate.

Sharpness can be proved with the following examples. Fix $n$, $k$ and $d$, and consider the family of manifolds $M_i$ defined by the warped product
\begin{gather}
 M_i = [-d/2,d/2]\times_{i^{-1} \tau_3} S^{n-1}\, ,
\end{gather}
where $S^{n-1}$ is the standard $n$-dimensional Riemannian sphere of radius $1$. It is easy to see that the diameter of this manifold satisfies
\begin{gather*}
 d<d(M_i)\leq \sqrt{d^2 + i^{-2}\pi^2 \tau_3(d/2)^2 }\, ,
\end{gather*}
and so it converges to $d$ as $i$ converges to infinity. Moreover, using standard computations it is easy to see that the Ricci curvature of $M_i$ is bounded below by $(n-1)k$ and that the boundary $\partial M_i = \{a,b\} \times S^{n-1}$ of the manifold is geodesically convex (for a detailed computation, see for example \cite[Section 5]{milman}).

As mentioned in Remark \ref{rem_w}, the function $u(t,x)= w_{3,d/2}(t) $ is a Neumann eigenfunction of the $p$-Laplace operator relative to the eigenvalue $\bar \lambda$. Since evidently the function $\bar \lambda (n,k,d)$ is continuous with respect to $d$, sharpness follows easily.
\end{proof}

\subsection*{Acknowledgments} We thank Dr. Songting Yin for pointing out the regularity issue explained in Remark \ref{rem_reg}.

\bibliographystyle{aomalpha}
\bibliography{nava_bib}

\def\cprime{$'$} \def\cprime{$'$} \def\cprime{$'$}
  \def\cftil#1{\ifmmode\setbox7\hbox{$\accent"5E#1$}\else
  \setbox7\hbox{\accent"5E#1}\penalty 10000\relax\fi\raise 1\ht7
  \hbox{\lower1.15ex\hbox to 1\wd7{\hss\accent"7E\hss}}\penalty 10000
  \hskip-1\wd7\penalty 10000\box7} \def\cprime{$'$}
\providecommand{\bysame}{\leavevmode\hbox to3em{\hrulefill}\thinspace}
\providecommand{\noopsort}[1]{}
\providecommand{\mr}[1]{\href{http://www.ams.org/mathscinet-getitem?mr=#1}{MR~#1}}
\providecommand{\zbl}[1]{\href{http://www.zentralblatt-math.org/zmath/en/search/?q=an:#1}{Zbl~#1}}
\providecommand{\jfm}[1]{\href{http://www.emis.de/cgi-bin/JFM-item?#1}{JFM~#1}}
\providecommand{\arxiv}[1]{\href{http://www.arxiv.org/abs/#1}{arXiv~#1}}
\providecommand{\doi}[1]{\url{http://dx.doi.org/#1}}
\providecommand{\MR}{\relax\ifhmode\unskip\space\fi MR }
\providecommand{\MRhref}[2]{%
  \href{http://www.ams.org/mathscinet-getitem?mr=#1}{#2}
}
\providecommand{\href}[2]{#2}
\begin{thebibliography}{RTU12}

\bibitem[BQ00]{new}
\bgroup\scshape{}D.~Bakry\egroup{} and \bgroup\scshape{}Z.~Qian\egroup{}, Some
  new results on eigenvectors via dimension, diameter, and {R}icci curvature,
  \emph{Adv. Math.} \textbf{155} (2000), 98--153. \mr{1789850}.
  \zbl{0980.58020}.  \doi{10.1006/aima.2000.1932}.

\bibitem[CW94]{chen-}
\bgroup\scshape{}M.~F. Chen\egroup{} and \bgroup\scshape{}F.~Y. Wang\egroup{},
  Application of coupling method to the first eigenvalue on manifold,
  \emph{Sci. China Ser. A} \textbf{37} (1994), 1--14. \mr{1308707}.
  \zbl{0799.53044}.

\bibitem[CW97]{chen}
\bgroup\scshape{}M.~Chen\egroup{} and \bgroup\scshape{}F.~Wang\egroup{},
  General formula for lower bound of the first eigenvalue on {R}iemannian
  manifolds,  \emph{Sci. China Ser. A} \textbf{40} (1997), 384--394.
  \mr{1450586}.  \zbl{0895.58056}.  \doi{10.1007/BF02911438}.

\bibitem[D{\v{R}}05]{dosly}
\bgroup\scshape{}O.~Do{\v{s}}l{\'y}\egroup{} and
  \bgroup\scshape{}P.~{\v{R}}eh{\'a}k\egroup{}, \emph{Half-linear differential
  equations}, \emph{North-Holland Mathematics Studies} \textbf{202}, Elsevier
  Science B.V., Amsterdam, 2005. \mr{2158903}.  \zbl{1090.34001}.

\bibitem[ENT13]{carlo}
\bgroup\scshape{}L.~Esposito\egroup{}, \bgroup\scshape{}C.~Nitsch\egroup{}, and
  \bgroup\scshape{}C.~Trombetti\egroup{}, Best constants in {P}oincar\'e
  inequalities for convex domains,  \emph{J. Convex Anal.} \textbf{20} (2013).
  Available at \url{http://arxiv.org/abs/1110.2960}.

\bibitem[FNT]{carlo2}
\bgroup\scshape{}V.~Ferone\egroup{}, \bgroup\scshape{}C.~Nitsch\egroup{}, and
  \bgroup\scshape{}C.~Trombetti\egroup{}, A remark on optimal weighted
  {P}oincar\'e inequalities for convex domains,  \emph{Atti Accad. Naz. Lincei
  Cl. Sci. Fis. Mat. Natur. Rend. Lincei (9) Mat. Appl.} (to appear). Available
  at \url{http://arxiv.org/abs/1207.0680}.

\bibitem[KN03]{kn}
\bgroup\scshape{}S.~Kawai\egroup{} and \bgroup\scshape{}N.~Nakauchi\egroup{},
  The first eigenvalue of the {$p$}-{L}aplacian on a compact {R}iemannian
  manifold,  \emph{Nonlinear Anal.} \textbf{55} (2003), 33--46. \mr{2001630}.
  \zbl{1043.53031}.  \doi{10.1016/S0362-546X(03)00209-8}.

\bibitem[Kr{\"o}92]{kro}
\bgroup\scshape{}P.~Kr{\"o}ger\egroup{}, On the spectral gap for compact
  manifolds,  \emph{J. Differential Geom.} \textbf{36} (1992), 315--330.
  \mr{1180385}.  \zbl{0738.58048}.  Available at
  \url{http://projecteuclid.org/euclid.jdg/1214448744}.

\bibitem[LF09]{wang}
\bgroup\scshape{}W.~Lin-Feng\egroup{}, Eigenvalue estimate of the
  {$p$}-{L}aplace operator,  \emph{Lobachevskii J. Math.} \textbf{30} (2009),
  235--242. \mr{2551865}.  \zbl{1223.53035}.  \doi{10.1134/S199508020903007X}.

\bibitem[Mat00]{matei}
\bgroup\scshape{}A.-M. Matei\egroup{}, First eigenvalue for the {$p$}-{L}aplace
  operator,  \emph{Nonlinear Anal.} \textbf{39} (2000), 1051--1068.
  \mr{1735181}.  \zbl{0948.35090}.  \doi{10.1016/S0362-546X(98)00266-1}.

\bibitem[Mil11]{milman}
\bgroup\scshape{}E.~Milman\egroup{}, Sharp isoperimetric inequalities and model
  spaces for curvature-dimension-diameter condition,  \emph{arXiv:1108.4609v2}
  (2011). Available at \url{http://arxiv.org/abs/1108.4609}.

\bibitem[Pet06]{petersen}
\bgroup\scshape{}P.~Petersen\egroup{}, \emph{{R}iemannian geometry}, second
  ed., \emph{Graduate Texts in Mathematics} \textbf{171}, Springer, New York,
  2006. \mr{2243772}.  \zbl{1220.53002}.

\bibitem[RTU12]{teix}
\bgroup\scshape{}J.~D. Rossi\egroup{}, \bgroup\scshape{}E.~V.
  Teixeira\egroup{}, and \bgroup\scshape{}J.~M. Urbano\egroup{}, Optimal
  regularity at the free boundary for the infinity obstacle problem,
  \emph{arXiv:1206.5652v1} (2012). Available at
  \url{http://arxiv.org/abs/1206.5652}.

\bibitem[Tol84]{reg}
\bgroup\scshape{}P.~Tolksdorf\egroup{}, Regularity for a more general class of
  quasilinear elliptic equations,  \emph{J. Differential Equations} \textbf{51}
  (1984), 126--150. \mr{727034}.  \zbl{0488.35017}.
  \doi{10.1016/0022-0396(84)90105-0}.

\bibitem[Val12]{svelto}
\bgroup\scshape{}D.~Valtorta\egroup{}, Sharp estimate on the first eigenvalue
  of the {$p$}-{L}aplacian,  \emph{Nonlinear Anal.} \textbf{75} (2012),
  4974--4994. \mr{2927560}.  \zbl{06057280}.  Available at
  \url{http://arxiv.org/abs/1102.0539}.

\bibitem[Wal98]{W}
\bgroup\scshape{}W.~Walter\egroup{}, Sturm-{L}iouville theory for the radial
  {$\Delta_p$}-operator,  \emph{Math. Z.} \textbf{227} (1998), 175--185.
  \mr{1605441}.  \zbl{0915.34022}.  \doi{10.1007/PL00004362}.

\bibitem[Zha07]{hui}
\bgroup\scshape{}H.~Zhang\egroup{}, Lower bounds for the first eigenvalue of
  the {$p$}-{L}aplace operator on compact manifolds with nonnegative {R}icci
  curvature,  \emph{Adv. Geom.} \textbf{7} (2007), 145--155. \mr{2290644}.
  \zbl{1121.58026}.  \doi{10.1515/ADVGEOM.2007.009}.

\end{thebibliography}

\end{document}